\documentclass[12pt, reqno]{amsart}
\usepackage[margin=0.75in]{geometry}
\usepackage[
  bookmarks=true]{hyperref}
\usepackage[OT2, T1]{fontenc}
\usepackage[english]{babel}
\usepackage[utf8]{inputenc}
\usepackage{csquotes}
\usepackage[final]{microtype}
\usepackage{lmodern}
\usepackage{amsthm}
\usepackage{amssymb}
\usepackage{mathrsfs}
\usepackage{enumerate}
\usepackage{tikz-cd} 
\usepackage{tikz}
\usepackage{multicol}
\usepackage{multirow}
\usepackage{tikz-qtree}
\usepackage{rotating}
\usepackage{graphicx}
\usepackage{comment}
\usepackage{enumitem}
\usepackage{manfnt}
\usepackage{colonequals}
\usetikzlibrary{arrows,calc,matrix,trees,arrows.meta,positioning,decorations.pathreplacing,bending}

\newtheorem{theorem}{Theorem}[section]

\newtheorem{lemma}[theorem]{Lemma}
\newtheorem{corollary}[theorem]{Corollary}

\theoremstyle{definition}

\newtheorem{remark}[theorem]{Remark}
\newtheorem{definition}[theorem]{Definition}

\newtheorem{question}[theorem]{Question}
\newtheorem{condition}[theorem]{Condition}

\newcommand{\Z}{\mathbb{Z}}
\newcommand{\N}{\mathbb{N}} 

\newcommand{\Q}{\mathbb{Q}}
\newcommand{\Oh}{\mathcal{O}} 

\newcommand{\Jac}{\mathrm{Jac}}
\newcommand{\Rank}{\mathrm{Rank}}
\newcommand{\tors}{\mathrm{tors}}

\newcommand{\genlegendre}[4]{%
  \genfrac{(}{)}{}{#1}{#3}{#4}%
  \if\relax\detokenize{#2}\relax\else_{\!#2}\fi
}

\DeclareSymbolFont{cyrletters}{OT2}{wncyr}{m}{n}
\DeclareMathSymbol{\Sha}{\mathalpha}{cyrletters}{"58}

\begin{document}

\title[Infinitely many hyperelliptic curves with fixed genus and positive rank]{Infinitely many hyperelliptic curves of small genus and small fixed rank, and of any genus and rank two}
\author{Stevan Gajovi\'c, Sun Woo Park}
\address{Max Planck Institute for Mathematics, Vivatsgasse 7, 53111 Bonn, Germany}
\email{\href{mailto:stevangajovic@gmail.com}{stevangajovic@gmail.com}, \href{mailto:swpark2008@gmail.com}{swpark2008@gmail.com}}

\begin{abstract}
We prove that for any number field $K$ and any fixed genus $g \geq 2$, there are infinitely many non-isomorphic hyperelliptic curves of genus $g$ over $K$ whose Jacobians have rank over $K$ equal to each of 0, 1, or 2. As an example of our method, over $\mathbb{Q}$, we prove that there exist infinitely many non-isomorphic hyperelliptic curves of genus two, whose Jacobians have rank equal to a fixed number between $1$ and $11$, genus three and four curves with rank between $1$ and $4$, and genus five and six with rank between $1$ and $3$.
\end{abstract}

\maketitle

\section{Introduction}\label{sec:intro}

Let $K$ be a number field, and $C/K$ a nice curve\footnote{smooth, projective, geometrically irreducible} of genus $g=g(C)$, with its Jacobian $\Jac(C)$. Mordell-Weil theorem states that $\Jac(C)(K)$ is a finitely generated abelian group, i.e., isomorphic to $\Z^r\times G$, where $G$ is a finite group. The number $r$ is called the rank of $\Jac(C)$ over $K$. 

Even though there are techniques to compute the rank $r$ in certain cases, for example, using descent \cite{Schaefer95, Stoll-Implementing-2-descent, SS03}, it is still a difficult problem to determine the exact value of $r$. In principle, the problem can be split into two tasks. The lower bound on $r$ can be obtained from finding enough $\Z$-linearly independent elements in $\Jac(C)(K)$. The upper bound on $r$ can be computed from using descent and computing Selmer groups of $\Jac(C)(K)$. When the two bounds are equal, $r$ can be explicitly determined. However, the bounds are often quite different, making computation of the rank difficult. While there are many families of curves whose ranks of Jacobians are known to be bounded by any side, there are much fewer examples of families of curves with exact ranks. 

Hence, for a concrete random curve, the determination of the rank of its Jacobian is a difficult problem. Now, we consider a more challenging question by studying ranks of Jacobians of families of curves, and in our paper, we consider hyperelliptic curves. We adopt the notation from \cite{KM25}: 

$$\mathcal{R}(K,g) = \left\{r \;\colon\; \substack{\text{there are infinitely many non-isomoprhic hyperelliptic curves } C \\ \text{of genus $g$ over $K$ such that } \Rank(\Jac(C)(K))=r}\right\}.$$

\medskip

In their recent paper \cite{KM25}, Koymans and Morgan proved a breakthrough result that $1\in \mathcal{R}(K,g)$ for any number field $K$ and any integer $g\geq 1$.
The main result of our paper, which we prove in Theorems~\ref{thm:numberfield-main} and~\ref{thm:any-g-rank-2} is that  $2\in \mathcal{R}(K,g)$ for any number field $K$ and any integer $g\geq 2$.

\begin{theorem}\label{thm:main-theorem-intro}
Let $g\geq 2$ be an integer and $K$ a number field. There are infinitely many non-isomorphic hyperelliptic curves $C$ of genus $g$ such that $\Rank(\Jac(C)(K))=2$.    
\end{theorem}

Until very recently, it was unknown if $2\in \mathcal{R}(\Q, 1)$, i.e., if there are infinitely many elliptic curves over $\Q$ of rank exactly 2. This was proven by Zywina in \cite{Zywina25}. In our paper, as a consequence of our constructions, we obtain a related but strategically different result; namely, in Corollary~\ref{cor:inifitely-many-g-2-r-0-4-over-Q}, we prove that $$\{0,1,2,3,4,5,6,7,8,9,10,11\}\subseteq \mathcal{R}(\Q, 2).$$

\subsection{Motivation}\label{subsec:our-motivation} The problem of studying ranks of certain families of elliptic curves has attracted considerable attention in past several years. There are also great mathematical merits of knowing the ranks of Jacobians of curves over number fields, as explained in more detail in \S\ref{subsec:importance-of-ranks}. 
There are many previous studies which analyse the statistics of rank over infinite family of curves, for example over universal families of elliptic curves \cite{BS15}, quadratic twist families of elliptic curves \cite{KMR14, Smith1, Smith3}, cubic twist families of elliptic curves \cite{ABS, KoymansSmith}, quartic twist families of elliptic curves over $\mathbb{Q}(i)$ \cite{Savoie25}, constructions of infinitely many elliptic curves of rank exactly equal to $1$ over any number field \cite{KP25, Zywina-2}, and infinitely many elliptic curves of fixed rank between 0 and 4 over any number field \cite{Zywina-3}.

More recently, there have been results about the analogous question of ranks of Jacobians of curves of genus $g\geq 2$, but this subarea of research is widely open. 
While there are some groundbreaking statistical results, such as the results by Yu \cite{Yu19} and Smith \cite{Smith2} who demonstrate infinitely many (respectively positive proportion) of quadratic twist of abelian varieties have rank $0$, infinite families of concrete rank remain pretty unknown, especially in the case of a positive rank. The only result we are aware of in this direction is the work by Koymans and Morgan on constructing infinitely many hyperelliptic curves of any genus $g$ \cite{KM25}, defined over any number field $K$, whose Jacobian has rank 1 over $K$. 
Our work can hence be considered as a continuation of the current progress to construct infinite families of hyperelliptic curves of any genus whose Jacobian has fixed positive rank.

\subsection{Our strategy}\label{subsec:our-strategy} Our novel and key idea is to intertwine \textbf{geometry of curves} with \textbf{arithmetic tools} to construct new examples of families of hyperelliptic curves of a fixed rank. To elaborate, we produce new hyperelliptic curves by \textit{glueing} or taking adequate \textit{fibre products} of two carefully chosen hyperelliptic curves. This strategy is discussed in detail in \S\ref{subsec:fibre-product}, where in the generic case the Jacobian of the new curve is isogenous to a product of Jacobians of three curves of smaller genus. However, as we discuss in \S\ref{subsec:higher-genus-strategy}, it is difficult to combine arithmetic tools with the generic case to produce the desired family of hyperelliptic curves. Instead, our choice of curves allow us to bypass this difficulty, by ensuring that the Jacobian of the new curve is isogenous over $K$ to a product of Jacobians of the previously chosen two curves of smaller genus. As far as we know, this geometric strategy has not been used in investigating this question before.

We can then proceed by quadratically twisting the two hyperelliptic curves to approach the problem. Our strategy of the proof comes in two flavours. Either we fix one of the two curves and quadratically twist the other, or we twist both curves with a good control on the rank of both Jacobians concurrently. 
To ensure the Jacobians of our new curves have fixed positive rank, we combine our geometric strategy with previous results regarding constructions of hyperelliptic curves whose Jacobian has smaller but fixed rank. One subtlety remains in taking this strategy, in that one needs to ensure that the technical conditions appearing in previous results are applicable to our setup. This is rigorously demonstrated in our article, see in particular \S\ref{subsec:constructions}.
Our article demonstrates that all relevant conditions appearing in such previous works can be checked and satisfied even after undertaking the constructions from the paper for two hyperelliptic curves over $K$. 

\noindent\textbf{Notation.} Throughout this manuscript, we will not write explicitly the adjective \emph{non-isomorphic} and the construction of infinitely many curves will always assume they are non-isomorphic. The proof that they are non-isomorphic in each case is straightforward, because we construct the curves by glueing or taking fibre products of twists of hyperelliptic curves, and for infinitely many twists, we introduce a new prime of bad reduction (using which we twist them with square-free elements).

\noindent\textbf{Thin family.} We explain why our careful geometric construction for studying ranks of Jacobians of hyperelliptic curves is a subtle non-trivial strategy. Recall that for $g\geq 2$, the dimensions of the moduli space of principally polarised abelian varieties of dimension $g$, the moduli space of curves of genus $g$, and the moduli space of hyperelliptic curves of genus $g$ over $\mathbb{C}$ are respectively given by $\frac{g(g+1)}{2}$, $3g - 3$, and $2g - 1$. For $g \geq 3$, it is clear that the hyperelliptic locus is a thin subset of both moduli spaces of principally polarised abelian varieties and moduli spaces of genus $g$ curves, so the strategy of simply taking products of lower-dimensional abelian varieties does not construct Jacobians of hyperelliptic curves. Even for $g = 2$, it is not clear whether a product of two elliptic curves is isogenous to the Jacobian of a hyperelliptic curve over a given number field $K$. These considerations highlight that our technique allows us to obtain rank statistics over a thin subfamily of the moduli space of principally polarised abelian varieties. 

\subsection{Brief summary of the content of the article}\label{subsec:content-intro}

In Section~\ref{sec:genus2}, we motivate our general construction of new hyperellitpic curves by first considering bielliptic \textit{genus two} curves $C$ such that $\Jac(C)$ is isogenous to the product of two elliptic curves $E_1$ and $E_2$. We divide the construction into 3 cases, depending on the 2-torsion subgroup of $E_1$ and $E_2$, which is defined over $K$.
\begin{itemize}
    \item In \S\ref{subsec:g2-no-2torsion}, we consider elliptic curves $E_1, E_2/\Q$ such that $E_1[2](\Q) \cong E_2[2](\Q) = 0$, and show that $r\in\mathcal{R}(\Q,2)$ for any $0 \leq r \leq 11$. In \S\ref{subsec:g2-Z2-2torsion} and \S\ref{subsec:g2-Z2Z2-2torsion}, in the case $E_1[2](\Q)\cong E_2[2](\Q) \neq 0$, we show $r \in\mathcal{R}(\Q,2)$ for any $0 \leq r \leq 4$ by constructing explicit infinite families of genus 2 curves and rank $r$. These examples are valuable because examples of families with explicit parametrisations are very rare; we only know for a few instances, \cite{Frey84, Monsky92, DV18}. 
    \item In \S\ref{subsec:rational-points}, we discuss rational points on these newly constructed genus $2$ hyperelliptic curves.
    \item In \S\ref{subsec:g2-Z2Z2-2torsion-any-number-field}, we show that $0,1,2\in\mathcal{R}(K,2)$ \textit{over any number field} $K$. These constructions use groundbreaking results in \cite{Smith2,KP25}, which prove $0\in \mathcal{R}(K,1)$ \cite{Smith2} and $1\in \mathcal{R}(K,1)$ \cite{KP25}. We also rigorously demonstrate the applicability of technical assumptions appearing in \cite{KP25} to our setup.
\end{itemize}

With the genus $2$ case explored, we proceed to Section \ref{sec:higher-genus} to consider \textit{hyperelliptic curves} $C$ \textit{of genus} $g \geq 3$ such that $\Jac(C)$ is isogenous to the product of two Jacobians of some carefully constructed hyperelliptic curves $C_1$ and $C_2$.
\begin{itemize}
    \item In \S\ref{subsec:fibre-product}, we discuss our geometric strategy of constructing a new hyperelliptic curve by taking fibre products of two hyperelliptic curves. By carefully choosing the two curves, we demonstrate in \S\ref{subsec:higher-genus-strategy} that the Jacobian of the new curve is isogenous to a product of Jacobians of two chosen curves of smaller genus.
    \item In \S\ref{subsec:constructions}, we show $0,1,2\in\mathcal{R}(K,g)$ for \textit{any number field} $K$ and \textit{any genus} $g\geq 3$. The overall strategy of the proof remains identical, but we employ different previous results to generate such curves \cite{Yu19, BM25, Smith2, KM25}. 
    Technical conditions appearing in \cite{Smith2} and \cite{KM25} are once again applicable to our construction of higher genus curves.
    \item In \S\ref{subsec:examples-3456-Q}, we further prove $3\in\mathcal{R}(\Q,g)$ for $3\leq g\leq 6$ and $4\in\mathcal{R}(\Q,g)$ for $3\leq g\leq 4$.
\end{itemize}

The breakthrough work by Koymans and Morgan \cite{KM25} in fact answers the previously posed challenge problem, which asks for generating infinitely many absolutely simple Jacobians of hyperelliptic curves over $\Q$ with rank $1$. We mention the formulation of this question in Section \ref{sec:challenges}.

The Magma \cite{Magma} code for our computations is available at \url{https://github.com/StevanGajovic/Ranks-Weil-Pairings}.

\subsection{Importance of $\mathbf{\textbf{Rank(}\textbf{Jac((C)(K)))}}$} \label{subsec:importance-of-ranks}
The problem of computing the rank of $\Jac(C)(K)$ is of high significance. We give some of its important consequences, as provided below.
\begin{itemize}
\item The rank of $\Jac(C)(K)$ is useful for studying the set $C(K)$, as one can use Chabauty and Coleman method \cite{Chabauty41, ColemanEC} to determine $C(K)$. Some generalisations of this method can be found in \cite{Kim2005MFG,Kim2009Selmer,QC13,SiksekSC}. These techniques have led to numerous results on studying modular curves that parametrise certain families of elliptic curves, their modularity, torsion subgroups, isogenies and residual Galois representations.
\item The proof of the Hilbert 10th problem over the ring of integers $\Oh_K$, where $K$ is a number field, crucially uses the construction of an abelian variety $A/K$ and a quadratic extension $L/K$ such that $\Rank(A(L))=\Rank(A(K))>0$. This was first proven in the groundbreaking work by Koymans and Pagano \cite{KP24} who constructed an elliptic curve $E/\Oh_K$ and a quadratic extension $L/K$ such that $\Rank(E(L))=\Rank(E(K))=1$. One can also construct infinitely many elliptic curves of rank exactly 1 which are not $\overline{K}$-isomorphic, as proven independently in \cite{Zywina-2}. A different proof of Hilbert 10th problem using positive rank of Jacobians of some twist families of hyperelliptic curves was shown by another groundbreaking work by Alp\"oge, Bhargava, Ho, and Shnidman \cite{ABHS25}.
\item The Birch-Swinnerton-Dyer conjecture \cite{BSD65} proposes intricate relations between algebraic properties and analytic properties of abelian varieties. Finding evidence supporting the conjecture has harbored fruitful results in understanding arithmetic properties of elliptic curves \cite{GZ86, Gross91, Zhang04, DLR15}. 
\end{itemize}

In the case of modular curves, using Eichler-Shimura relations, one may decompose their Jacobians and compute their rank by computing the rank of their components. Still, this technique is not feasible for most curves. Even for hyperelliptic curves, for which we can try to apply the 2-descent, it is not clear if the upper bound obtained from descent will match the number of $\Z$-linearly independent elements in $\Jac(C)(K)$. 
For example, in \cite{SharpCurves}, the first named author presents infinite families of curves over $\Q$ whose Jacobians have large rank. Namely if we denote by $g(C)$ the genus of $C$, then an infinite family of hyperelliptic curves such that $\Rank(\Jac(C)(\Q))\geq g(C)-1$ is given in \cite[Theorem 12(1)]{SharpCurves}, whereas $\Rank(\Jac(C)(\Q))\geq g(C)$ is given in \cite[Theorem 12(2)]{SharpCurves}. The main idea was to find curves with too many rational points that would violate Coleman's bound \cite{ColemanEC}, or Stoll's improvement \cite{StollImprovedBound}. However, giving an exact rank of such curves seems extremely difficult.

\section{Genus two curves}\label{sec:genus2}

We consider bielliptic curves $C$ of genus two; then $\Jac(C)\sim E_1\times E_2$. In our constructions we fix one curve, say $E_1$, and we vary $E_2$. Recall that, in general, $E_1\times E_2$ is not isogenous to a Jacobian of a genus two curve over $K$, but in our cases, the curves $E_1$ and $E_2$ are ``compatible'' so it is. More formally, we can glue them so that one gets a bielliptic curve of genus 2, whose Jacobian is isogenous to $E_1\times E_2$. We provide the constructions for all possible rational two torsion subgroups when considered over $\Q$, i.e., $E_1[2](\Q)\cong E_2[2](\Q)\in \{\langle O\rangle, \Z/2\Z, \Z/2\Z\times \Z/2\Z\}$. However, our construction in \S\ref{subsec:g2-no-2torsion} works over more general number fields.

\subsection{No rational two-torsion}\label{subsec:g2-no-2torsion}

\begin{theorem} \label{thm:infinitely-many-g-2-no-2torsion}
Let $K$ be a number field and $r\in \Z_{\geq 0}$. Assume that there is an elliptic curve $E\colon y^2=x^3+a^2$, where $a\in\Oh_K\setminus \Oh_K^3$ such that $\Rank(E(K))=r$ and $K(E[2])/K$ is an $S_3$ extension. There are infinitely many genus two curves $C_{a,m}\colon y^2=x^6+m^3a$ defined over $K$ such that $\Rank(\Jac(C)(K))=r$.
\end{theorem}
\begin{proof}
We first construct hyperelliptic curves of genus two which are also bielliptic. Given any $D \in K^\times$, consider the three curves over $K$:
\begin{align}
\begin{split}
    C_D &\colon y^2 = x^6 + D,\\
    E_{1,D} &\colon y^2 = x^3 + D, \\
    E_{2,D} &\colon y^2 = x^3 + D^2.
\end{split}
\end{align}
Then there exists an isogeny $\phi \colon \Jac(C_D) \to E_{1,D} \times E_{2,D}$. A direct application of \cite[Theorem 3.2]{BD11} shows that $\Jac(C_D)$ is isogenous to a product of two elliptic curves whose Weierstrass equations are $E'_{1,D}\colon y^2 = x^3+ D$ and $E'_{2,D}\colon y^2 = x^4 + Dx$. A change of variables given by $X := \frac{D}{x}$ and $Y := \frac{Dy}{x^2}$ gives the Weierstrass equation for $E_{2,D}$.

Using these curves, we construct our desired infinite family of hyperelliptic curves as follows. Fix an element $a \in \Oh_K\setminus \Oh_K^3$ satisfying the conditions of the statement of the theorem. Let $\mathcal{D}_a$ be the collection of elements in $\Oh_K$ defined as
\begin{equation}
    \mathcal{D}_a := \left\{ a \cdot m^3 \in \Oh_K \colon m \in \Oh_K^\times \text{ is square-free} \right\}.
\end{equation}
Then for any $D \in \mathcal{D}_a$, the Weierstrass equation for $E_{2,D}$ can be rewritten as $E \colon y^2 = x^3 + a^2$, which is independent of the choice of $D$. Because $K(E[2])/K$ is an $S_3$ extension, $K(E_{1,D}[2])/K$ is also an $S_3$ extension. We note that $\{E_{1,D}\}_{D \in D_a}$ is a family of quadratic twists of the elliptic curve $E'\colon y^2 = x^3 + a$. By \cite[Theorem A]{KMR14}, there exists an infinite subset $\mathcal{D}^0_a \subset \mathcal{D}_a$ such that $\dim_{\mathbb{F}_2} \mathrm{Sel}_2(E_{1,D}/K) = 0$, hence $\Rank E_{1,D}(K) = 0$, for every $D \in \mathcal{D}^0_a$. Hence, we have that $\{C_D\}_{D \in D^0_a}$ is an infinite set of hyperelliptic curves such that $\Rank(\Jac(C_D)(K)) = r$. 
\end{proof}

\begin{corollary}\label{cor:inifitely-many-g-2-r-0-4-over-Q}
Let $0\leq r\leq 11$ be a fixed integer. There are infinitely many genus two curves $C$ defined over $\Q$, such that $\Rank(\Jac(C)(\Q))=r$.
\end{corollary}
\begin{proof}
For $0 \leq r \leq 4$, 
we apply Theorem~\ref{thm:infinitely-many-g-2-no-2torsion} for the following values of $a=2,3,15,427,13\cdot19\cdot23\cdot43$ for $r=0,1,2,3,4$, respectively. The corresponding rank can easily be checked in Magma \cite{Magma}.  For $5 \leq r \leq 11$, we apply the main result from \cite{ER04}. We note that the elliptic curve $y^2 = x^3 - 432 k^2$ is $3$-isogenous to the elliptic curve $y^2 = x^3 + 16k^2$ \cite[Section 1]{CP09}. We can then set $a = 4k$ for the following values of $k$ to produce infinitely many genus two curves of rank $r$ as provided below.
\begin{align}
    \begin{split}
        r = 5 &: k = 3 \cdot 7 \cdot 11 \cdot 13 \cdot 163,\\
        r = 6 &: k = 3 \cdot 73 \cdot 103 \cdot 439,\\
        r = 7 &: k = 3 \cdot 13 \cdot 19 \cdot 41 \cdot 139 \cdot 271,\\
        r = 8 &: k = 2 \cdot 3 \cdot 5 \cdot 7 \cdot 11 \cdot 13 \cdot 17 \cdot 29 \cdot 41 \cdot 47 \cdot 59,\\
        r = 9 &: k = 2 \cdot 5 \cdot 37 \cdot 41 \cdot 53 \cdot 73 \cdot 1231 \cdot 4831,\\
        r = 10 &: k = 2 \cdot 3 \cdot 5 \cdot 7 \cdot 23 \cdot 31 \cdot 37 \cdot 43 \cdot 83 \cdot 109 \cdot 151 \cdot 421,\\
        r = 11 &: k = 3 \cdot 5 \cdot 7 \cdot 13 \cdot 19 \cdot 23 \cdot 31 \cdot 43 \cdot 59 \cdot 61 \cdot 73 \cdot 79 \cdot 103 \cdot 109 \cdot 157 \cdot 457.
    \end{split}
\end{align}
\end{proof}

\begin{remark}\label{rmk:condition-E[2]}
The condition of $K(E[2])/K$ being an $S_3$ extension amounts to $\zeta_3\notin K$ (necessary condition) and $\sqrt[3]{a}\notin K$. Therefore, for all number fields $K$ such that $3\nmid [K:\Q]$ and $\zeta_3\notin K$, this condition is satisfied. Otherwise, for fields with $\zeta_3\notin K$, but with $\sqrt[3]{a}\in K$, one can try to find another elliptic curve $E'\colon y^2=x^3+a'^2$ satisfying the conditions of Theorem~\ref{thm:infinitely-many-g-2-no-2torsion}. 
\end{remark}

\begin{remark}\label{rmk:smith}
Suppose $K = \mathbb{Q}$. Using the main results from \cite{Smith1, Smith2, Smith3}, we can nullify the condition from Theorem \ref{thm:infinitely-many-g-2-no-2torsion} that $K(E[2])/K$ is an $S_3$ extension by allowing $a$ to be a cube. Let $a \in \mathbb{Q}^\times$ be a fixed number. As before, let $E: y^2 = x^3 + a^2$. Then at least $50\%$ of genus two curves in $\{C_D\}_{D \in \mathcal{D}_a}$ have Jacobians of rank equal to $\Rank(E(\Q))$, and $100\%$ of genus two curves in $\{C_D\}_{D \in \mathcal{D}_a}$ have Jacobians of rank at most $\Rank(E(\Q))+1$. 
\end{remark}

\subsection{Partial rational two-torsion}\label{subsec:g2-Z2-2torsion}

Now we consider the family of bielliptic curves $C_{m,d}\colon y^2=d^3x^6+m^3$, where $d,m
\in \Z$ are square-free. Similarly as before, \cite[Theorem 3.2]{BD11} implies that $\Jac(C_{m,d})$ is isogenous to $E_{d}\times E_m$, where $E_d\colon y^2=x^3+d^3$ and $E_m\colon y^2=x^3+m^3$, which is a quadratic twist of $E_d$ by $\frac{m}{d}$. Now, we fix $d$ and vary $m$, obtaining the following theorem.

\begin{theorem}\label{thm:infinitely-many-g-2-Z2-2torsion}
Let $r\in \Z_{\geq 0}$ and $d\in \Z$ be such that $\Rank(E_d(\Q))=r$, where $E_d\colon y^2=x^3+d^3$. Then for the following curves of genus two $C_{p,d}\colon y^2=d^3x^6+p^3$, where $p$ is a prime number, we have
\begin{itemize}
\item if $p\equiv 5\pmod{12}$, $\Rank(C_{p,d}(\Q))=r$;
\item if $p\equiv 3\pmod{4}$ and $p>3$, $\Rank(C_{p,d}(\Q))=r+1$;
\end{itemize}
\end{theorem}

\begin{proof}
The proof follows from the text above and the fact that $\Rank(E_p(\Q))=0,1$, when $p\equiv 5\pmod{12}$, $p\equiv 3\pmod{4}$ and $p>3$, respectively, see \cite[P. 73, Satz 2, 3]{Frey84}.
\end{proof}

We note that this approach works only over $\Q$, but unlike the previous approach, it gives \emph{explicit} families of infinitely many curves of fixed positive rank.
A quick Magma search gives us the following rank information:
\begin{itemize}
\item $d=1$: $E_1\colon y^2=x^3+1$, $\Rank(E_1(\Q))=0$;
\item $d=2$: $E_2\colon y^2=x^3+2^3$, $\Rank(E_2(\Q))=1$;
\item $d=37$: $E_{37}\colon y^2=x^3+37^3$, $\Rank(E_{37}(\Q))=2$;
\item $d=506$: $E_{506}\colon y^2=x^3+506^3$, $\Rank(E_{506}(\Q))=3$.
\end{itemize}
For example, $y^2=506^3x^6+p^3$, where $p\equiv 3\pmod{4}$ and $p>3$, is a family of curves whose Jacobian has rank 4, whereas if $p\equiv 5\pmod{12}$, then this is a family of curves whose Jacobian has rank 3. 

\subsection{Full rational two-torsion - over $\Q$}\label{subsec:g2-Z2Z2-2torsion}

Now we consider elliptic curves $E_1,E_2/\Q$ such that $E_1[2](\Q)\cong E_2[2](\Q)\cong \Z/2\Z\times \Z/2\Z$. Since we want to vary one of them, we will specify one of them to be a congruent problem elliptic curve, i.e., let $E_1$ belong to the family $E^p\colon y^2=x^3-p^2x$, where $p\equiv 3\pmod{8}$; then by \cite{Nagell29}, we have $\Rank(E^p(\Q))=0$. We want to use \cite[Theorem 3.2]{BD11}, so, we look at the other elliptic curve of the form ($k\neq -1,0,1$)
$$E_2\colon y^2=d(x-kp)(x^3-p^2x),$$
which after replacing $(x,y)$ by $(px,p^2y)$ becomes $E_{2,d,k}\colon y^2=d(x-k)(x^3-x)$. From now on, we can search directly in Magma for curves $E_{2,d,k}$ of any rank we want. Still, if we prefer to use the standard way of expressing elliptic curves, a rational change of variables leads to the following model 
$$E'_{2,d,k}\colon dy^2=(x+k^2-k)(x+k^2-1)(x+k^2+k).$$
We obtain the following theorem, which immediately follows from \cite[Theorem 3.2]{BD11}.

\begin{theorem}\label{thm:infinitely-many-g-2-Z2Z2-2torsion}
Let $r\in \Z_{\geq 0}$, $k\in \Z\setminus \{-1,0,1\}$, and square-free $d\in \Z$ be such that $\Rank(E'_{2,d,k}(\Q))=r$, where $E'_{2,d,k}\colon dy^2=(x+k^2-k)(x+k^2-1)(x+k^2+k)$. Then, for the following curves of genus two 
$$C_{d,k,p}\colon dy^2=x^6+3dkpx^4+(3k^2p^2-1)d^2x^2+(k^3p^3-kp)d^3,$$ where $p\equiv 3\pmod{8}$ is a prime number, we have $\Rank(\Jac(C_{d,k,p})(\Q))=r$.  
\end{theorem}

Testing the curves $E'_{2,d,k}$ for $d=1$ in Magma, we obtain the following rank information:
\begin{itemize}
\item $k=2$: $\Rank(E'_{2,1,2}(\Q))=0$;
\item $k=3$: $\Rank(E'_{2,1,3}(\Q))=1$;
\item $k=11$: $\Rank(E'_{2,1,11}(\Q))=2$;
\item $k=43$: $\Rank(E'_{2,1,43}(\Q))=3$;
\item $k=329$: $\Rank(E'_{2,1,329}(\Q))=4$.
\end{itemize}
For example, one more \emph{explicit} family of curves of genus two whose Jacobians have rank 4 is the family $C_{1,329,p}$, where $p\equiv 3\pmod{8}$. 

\subsection{Rational points on these curves}\label{subsec:rational-points} 
In this section, we only consider curves defined over $\Q$, so we exclude those from \S\ref{subsec:g2-Z2Z2-2torsion-any-number-field}. Even though we construct curves of higher rank, it is easy to determine their rational points in most of our examples (except for $C_{p,d}$ from \S\ref{subsec:g2-Z2-2torsion} when $p\equiv 3\pmod{4}$) because these curves admit a quotient map to a rank 0 elliptic curve. This elliptic curve has only finitely many points, and we can determine which pullbacks of these points are rational points. Now we focus on all but finitely many curves, which could still have more rational points.

For example, curves from \S\ref{subsec:g2-no-2torsion} and \S\ref{subsec:g2-Z2-2torsion} admit a quotient map to an elliptic curve of the form $E\colon y^2=x^3+s$, $s\in\Z$. By a well-known result \cite[Exercise 10.19]{Silverman}, when $s\neq -432$ is not a square or a cube in $\Z$, then $E_{\tors}=\langle O\rangle$. This is the case for almost all curves in \S\ref{subsec:g2-no-2torsion}, so $C_{a,m}(\Q)=\{\infty_{\pm}\}$ (with projective coordinates $(1:\pm1:0)$). 

On the other hand, also by \cite[Exercise 10.19]{Silverman}, for $E_p\colon y^2=x^3+p^3$ with $p\equiv 5\pmod{12}$, we have $E_p(\Q)=\{O,(-p,0)\}$. Hence, if $d$ in a square (i.e. $d=1$ since we assume $d$ is square-free), then $C_{p,d}(\Q)=\{\infty_{\pm}\}$. Otherwise, the set $C_{p,d}(\Q)$ is empty. When $p\equiv 3\pmod{4}$, the problem of determining $C_{p,d}(\Q)$ could be very difficult.

In \S\ref{subsec:g2-Z2Z2-2torsion}, the structure of the $\mathbb{Q}$-rational points of rank 0 elliptic curve $E^p\colon y^2=x^3-p^2x$ is well known: $E^p(\Q)=E^p(\Q)[2]\cong\Z/2\Z\times \Z/2\Z$. Pullback of nontrivial two-torsion points to $C_{d,k,p}$ would give an $x$-coordinate satisfying $x^2\in\{-dkp,-d(k-1)p,-d(k+1)p\}$, but then $x\notin \Q$ for almost all $p$. Hence, if $d=1$, then we have $C_{d,k,p}(\Q)=\{\infty_{\pm}\}$ almost always. Otherwise, we have $C_{d,k,p}(\Q)=\emptyset$.

\subsection{Full rational two-torsion - over any number field}\label{subsec:g2-Z2Z2-2torsion-any-number-field}
Let $K$ be any number field. We now consider elliptic curves $E_1, E_2 / K$ such that $E_1[2](K) \cong E_2[2](K) \cong \mathbb{Z}/2\mathbb{Z} \times \mathbb{Z}/2\mathbb{Z}$. Unlike the case where $K = \mathbb{Q}$, we will vary both elliptic curves over their families of quadratic twists. Let $a_1, a_2, a_3 \in \mathcal{O}_K \setminus \{0\}$. We consider the following elliptic curve, where $m$ is any square-free element of $\mathcal{O}_K$.
\begin{align*}
    E_{1,m} &: y^2 = m(x-a_1)(x-a_2)(x-a_3)
\end{align*}
We then choose $\alpha \in \mathcal{O}_K \setminus \{a_1,a_2,a_3\}$ and a square-free element $d \in \mathcal{O}_K$ to define the following elliptic curve by using \cite[Theorem 3.2]{BD11}:
\begin{equation*}
    E_{2} : y^2 = d(x-m\alpha)(x-ma_1)(x-ma_2)(x-ma_3).
\end{equation*}
After replacing $(x,y)$ with $(mx,m^2y)$, we obtain $E_{2,d}: y^2 = d(x-\alpha)(x-a_1)(x-a_2)(x-a_3)$. We can summarize our findings as stated in the following theorem.
\begin{theorem} \label{thm:infinitely-many-g-2-Z2Z2-2torsion-numberfield}
Let $r_1, r_2 \in \mathbb{Z}_{\geq 0}$, $a_1,a_2,a_3 \in \mathcal{O}_K$, $\alpha \in \mathcal{O}_K \setminus \{a_1,a_2,a_3\}$, and $m, d \in \mathcal{O}_K$ be square-free. Let $E_{1,m}: y^2 = m(x-a_1)(x-a_2)(x-a_3)$ and $E_{2,d}: y^2 = d(x-\alpha)(x-a_1)(x-a_2)(x-a_3)$ be elliptic curves such that $\Rank(E_{1,m}(K)) = r_1$ and $\Rank(E_{2,d}(K)) = r_2$. Then for the following curve of genus two
\begin{equation*}
    C_{m,d}: y^2 = m\left(\frac{1}{d}x^2 + \alpha - a_1 \right)\left(\frac{1}{d}x^2 + \alpha - a_2 \right)\left(\frac{1}{d}x^2 + \alpha - a_3 \right)
\end{equation*}
we have $\Rank(C_{m,d}(K)) = r_1 + r_2$.
\end{theorem}
Combining the above theorem with recent breakthroughs in distribution of algebraic rank of quadratic twist families of elliptic curves over number fields \cite{Smith2, KP25}, we can prove the following result for genus 2 curves over any number field $K$.
\begin{theorem} \label{thm:numberfield-main}
    Let $K$ be any number field. Let $0 \leq r \leq 2$ be a fixed integer. Then there are infinitely many genus two curves $C$ defined over $K$ such that $\Rank(Jac(C)(K)) = r$.
\end{theorem}
\begin{proof}
We choose nine prime elements $p_i, q_i, r_i$ of $\Oh_K$ so that the following conditions are satisfied:
   \begin{itemize}
       \item $p_i, q_j, r_k$ are totally positive.
       \item $p_i, q_j, r_k \equiv 1 \mod 8 \mathcal{O}_K$.
       \item $\# \mathcal{O}_K/(p_i), \# \mathcal{O}_K/(q_i), \# \mathcal{O}_K/(r_i) > 5$.
       \item For every $1 \leq i, j, k \leq 3$, we have $\left(\dfrac{p_i}{q_j}\right) = \left(\dfrac{q_j}{r_k}\right) = \left(\dfrac{r_k}{p_i}\right) = 1$.
   \end{itemize}

We will use the following result which is an immediate consequence of \cite[P. 20, Theorem]{Mitsui}
\begin{lemma}\label{lem:totally-positive-primes-in-progressions}
Let $I$ be an ideal of $K$ and $\rho\in \Oh_K$ be totally positive such that the ideals $I$ and $(\rho)$ are coprime. Then there are infinitely many totally positive prime elements of $p\in K$ satisfying $p\equiv \rho\pmod{I}$.   
\end{lemma}

Now we apply Lemma~\ref{lem:totally-positive-primes-in-progressions} to the ideal $I=(8)$ and $\rho=1$ to find totally positive prime elements $p_1,p_2,p_3\in \Oh_K$ such that $\#\mathcal{O}_K/(p_i)>5$. Then, we apply Lemma~\ref{lem:totally-positive-primes-in-progressions} to the ideal $I=(8p_1p_2p_3)$ and $\rho=1$ to find totally positive prime elements $q_1,q_2,q_3\in \Oh_K$ such that $\#\mathcal{O}_K/(q_i)>5$. Finally, we apply Lemma~\ref{lem:totally-positive-primes-in-progressions} to the ideal $I=(8p_1p_2p_3q_1q_2q_3)$ and $\rho=1$ to find totally positive prime elements $r_1,r_2,r_3\in \Oh_K$ such that $\#\mathcal{O}_K/(r_i)>5$. By construction, we have that $\left(\dfrac{p_i}{q_j}\right) = \left(\dfrac{q_j}{r_k}\right) = \left(\dfrac{r_k}{p_i}\right) = 1$.\\

For such a choice of $p_i, q_j, r_k$, we define $P$, $Q$, $R$ and $a_1, a_2, a_3$ as follows:
    \begin{equation*}
        P := p_1 p_2 p_3, \hspace{5pt} Q := q_1 q_2 q_3, \hspace{5pt} R := r_1 r_2 r_3, \hspace{5pt} a_1 := P^2 Q, \hspace{5pt} a_2 := Q^2 R, \hspace{5pt} a_3 := R^2 P.
    \end{equation*}
    
    Let $m,d$ be square-free elements of $\mathcal{O}_K$. We consider the family of genus $2$ curves $C_{m,d}$ defined as
    \begin{equation*}
        C_{m,d}: y^2 = m \left( \frac{1}{d}x^2 - a_1 \right) \left( \frac{1}{d}x^2 - a_2 \right) \left( \frac{1}{d}x^2 - a_3 \right),
    \end{equation*}
    which is isogenous to a product of two elliptic curves provided below:
    \begin{align*}
        E_{1,m} &: y^2 = m(x-a_1)(x-a_2)(x-a_3), \\
        E_{2,d} &: y^2 = d(x-a_2a_3)(x-a_1a_3)(x-a_1a_2).
    \end{align*}
    The Weierstrass equation for $E_{2,d}$ is obtained as follows. We set $\alpha = 0$ and apply Theorem \ref{thm:infinitely-many-g-2-Z2Z2-2torsion-numberfield} to get $E_2: y^2 = dx(x-a_1)(x-a_2)(x-a_3)$. We then apply the rational change of coordinates which replaces $(x,y)$ with $(-\frac{1}{x} \cdot {a_1 a_2 a_3}, \frac{y}{x^2} \cdot a_1 a_2 a_3)$ to get the Weierstrass equation for $E_{2,d}$.

    Notice that the four conditions on nine prime elements $p_i, q_j, r_k$ imply the following (because they all have odd valuation with respect to at least one of the primes $p_i, q_j, r_k$). 
    \begin{itemize}
        \item None of $(a_1-a_2)(a_1-a_3)$, $(a_2-a_1)(a_2-a_3)$, and $(a_3-a_1)(a_3-a_2)$ are squares in $K$.
        \item None of $a_2a_3(a_2-a_1)(a_3-a_1)$, $a_1a_3(a_1-a_2)(a_3-a_2)$, and $a_1a_2(a_1-a_3)(a_2-a_3)$ are squares in $K$.
    \end{itemize}
    By \cite[Theorem 2.14, Example 3.1]{Smith2}, there exists infinitely many (in fact a positive proportion of) square-free elements $m, d \in \mathcal{O}_K$ such that $\Rank(E_{1,m}(K)) = \Rank(E_{2,d}(K)) = 0$. Here, we note that there is a subtlety with generating infinitely many such rank 0 quadratic twist families of elliptic curves over any number field $K$. This is where the ``parity-invariant subcase'' appearing after \cite[Case 2.13]{Smith2} becomes relevant. (We would like to sincerely thank Adam Morgan for pointing out this subtlety.) Our construction of elliptic curves $E_{1,m}$ and $E_{2,d}$ shows that both curves have at least one place of multiplicative reduction. By the first paragraph of proof of \cite[Proposition 4.3]{MS24} and \cite[Lemma 4.2]{MS24}, performing an unramified quadratic twist at the local condition over such a multiplicative place changes the parity of 2-Selmer group of $E_{1,m}$ (and likewise for $E_{2,d}$). Hence, we can always find infinitely many $m, d \in \mathcal{O}_K$ (in fact a positive density of square-free $m,d \in \mathcal{O}_K$) such that the parity of 2-Selmer group of $E_{1,m}$ is even. By \cite[Theorem 2.14, Example 3.1]{Smith2}, there are infinitely many choices of $m,d \in \mathcal{O}_K$ such that $\Rank(\Jac(C_{m,d})(K)) = 0$.
    
    We turn our focus to constructing infinitely many curves $C_{m,d}$ such that $\Rank(\Jac(C_{m,d})(K)) = 1$. We note that this result for $g = 2$ was obtained before \cite{KM25}. By our choice of nine prime elements $p_i, q_j, r_k$ the elliptic curve $E_{1,1}$ satisfies the following three conditions:
    \begin{itemize}
        \item $a_1 - a_2$ has odd valuation with respect to primes $q_1, q_2$, and $q_3$, whereas $a_1 - a_3$ and $a_2-a_3$ are invertible squares modulo $q_1, q_2$, or $q_3$.
        \item $a_2 - a_3$ has odd valuation with respect to primes $r_1, r_2$, and $r_3$, whereas $a_2 - a_1$ and $a_3-a_1$ are invertible squares modulo $r_1, r_2$, or $r_3$.
        \item $a_3 - a_1$ has odd valuation with respect to primes $p_1, p_2$, and $p_3$, whereas $a_3 - a_2$ and $a_1-a_2$ are invertible squares $p_1, p_2$, or $p_3$.
    \end{itemize}
    These three conditions are in fact the ``3-generic'' conditions appearing in \cite[Definition 1.2]{KP25}. By \cite[Theorem 1.3]{KP25}, there exist infinitely many square-free elements $m$ such that $\Rank(E_{1,m}(K)) = 1$. We can again use \cite[Theorem 2.14, Example 3.1]{Smith2} to obtain that there are infinitely many square-free elements $d \in \mathcal{O}_K$ such that $\Rank(E_{2,d}(K)) = 0$. For such $m$ and $d$ we have $\Rank(\Jac(C_{m,d})(K)) = 1$.

    Lastly, we construct infinitely many curves $C_{m,d}$ such that $\Rank(\Jac(C_{m,d})(K)) = 2$. 
 We still choose nine totally positive prime elements $p_i, q_j, r_k\in \Oh_K$ in the same manner.

 We also choose the same $a_1,a_2,a_3$, but now we change $\alpha$. Consider $\alpha \in \Oh_K$ such that $\alpha$ is a quadratic residue modulo all $p_i,q_j,r_k$, but such that $\alpha\not\equiv a_1\pmod{r_i}$,  $\alpha\not\equiv a_2\pmod{p_i}$, and $\alpha\not\equiv a_3\pmod{q_i}$ for $1\leq i\leq 3$, by the assumption that all residue fields have cardinality greater than 5, we can certainly satisfy the conditions modulo each prime, and then by the Chinese Remainder Theorem obtain $\alpha$. 

   For such a choice of $p_i, q_j, r_k,$ and $\alpha$, we consider the family of genus 2 curves $C_{m,d}$ defined as
   \begin{equation*}
       C_{m,d} : y^2 = m \left( \frac{1}{d}x^2 + \alpha - a_1 \right)\left( \frac{1}{d}x^2 + \alpha - a_2 \right)\left( \frac{1}{d}x^2 + \alpha - a_2 \right)
   \end{equation*}
   which is isogenous to a product of two elliptic curves provided below:
   \begin{align*}
       E_{1,m} &: y^2 = m(x-a_1)(x-a_2)(x-a_3), \\
       E_{2,d} &: y^2 = d(x + (a_2-\alpha)(a_3-\alpha))(x+(a_1-\alpha)(a_3-\alpha))(x + (a_1-\alpha)(a_2-\alpha)).
   \end{align*}
   As before, our choices of $p_i, q_j, r_k$ imply that $E_{1,1}$ is $3$-generic. The choice of $\alpha$ implies that the following three conditions hold:
\begin{itemize}
    \item $(\alpha-a_3)(a_1 - a_2)$ has odd valuation with respect to primes $q_1, q_2$, and $q_3$, whereas $(\alpha-a_2)(a_1 - a_3)$ and $(\alpha - a_1)(a_2 - a_3)$ are invertible squares modulo $q_1, q_2$, or $q_3$.
    \item $(\alpha-a_1)(a_2-a_3)$ has odd valuation with respect to primes $r_1, r_2$, and $r_3$, whereas $(\alpha-a_3)(a_2 - a_1)$ and $(\alpha - a_2)(a_3-a_1)$ are invertible squares modulo $r_1, r_2$, or $r_3$.
    \item $(\alpha-a_2)(a_3 - a_1)$ has odd valuation with respect to primes $p_1, p_2$, and $p_3$, whereas $(\alpha-a_1)(a_3 - a_2)$ and $(\alpha - a_3)(a_1-a_2)$ are invertible squares modulo $p_1, p_2$, or $p_3$.
\end{itemize}
These three conditions show that $E_{2,1}$ is 3-generic. By \cite[Theorem 1.3]{KP25}, there exist infinitely many square-free elements $m$ and $d$ such that $\Rank(E_{1,m}(K)) = 1$ and $\Rank(E_{1,d}(K)) = 1$. For such $m$ and $d$ we have $\Rank(\Jac(C_{m,d})(K)) = 2$.
\end{proof}

\section{Higher genus curves} \label{sec:higher-genus}

We now use a more general construction than the method of glueing two elliptic curves explored in Section~\ref{sec:genus2}, namely, we produce hyperelliptic curves, which are the fibre product of two hyperelliptic curves. This allows us to construct infinitely many hyperelliptic curves of any genus over any number field whose Jacobians have fixed positive rank between $0$ and $2$. We also explore examples of genus 3, 4, 5, and 6 hyperelliptic curves over $\mathbb{Q}$ with fixed small positive rank.

\subsection{Fibre product of hyperelliptic curves}\label{subsec:fibre-product}

Let $K$ be a number field, $f,g\in K[x]$ different squarefree polynomials, and consider hyperelliptic curves $C_f\colon y^2=f(x)$ and $C_g\colon y^2=g(x)$. Then, there are double covers  $\pi_f\colon C_f\longrightarrow\mathbb{P}^1$ and $\pi_g\colon C_g\longrightarrow\mathbb{P}^1$. Hence, there is a curve $C/K$ whose function field $K(C)$ is the compositum $K(C_f) \otimes_{K(x)} K(C_g)$. If we denote $L=K(x)$ (the function field of $\mathbb{P}^1$), then $K(C_f) = L(\sqrt{f(x)})$ and $K(C_g) = L(\sqrt{g(x)})$,  and it follows that $K(C) = L(\sqrt{f(x)}, \sqrt{g(x)})$. Then, $C$ is the fibre product of $C_f$ and $C_g$ over $\mathbb{P}^1$. Furthermore, we note that $K(C) = L(\sqrt{f(x)}, \sqrt{g(x)})=L(\sqrt{f(x)}, \sqrt{f(x)g(x)})=L(\sqrt{f(x)g(x)}, \sqrt{g(x)})$. So, if we denote $h(x)$ to be the squarefree part of $f(x)g(x)$, then $C$ can also be viewed as the fibre product of $C_f$ and $C_h\colon y^2=h(x)$ over $\mathbb{P}^1$, and similarly also as the fibre product of $C_g$ and $C_h$.

Using the functoriality of Jacobians, the maps $\pi_f, \pi_g, \pi_h: C \longrightarrow C_f, C_g, C_h$ induce surjective homomorphisms $\psi_f: \mathrm{Jac}(C) \longrightarrow \mathrm{Jac}(C_f)$, $\psi_g: \mathrm{Jac}(C) \longrightarrow \mathrm{Jac}(C_g)$, and $\psi_h: \mathrm{Jac}(C) \longrightarrow \mathrm{Jac}(C_h)$, see \cite[p.328]{MUMFORD1974325} for a proof for surjectivity of these maps. We obtain a homomorphism $\psi=(\psi_f,\psi_g,\psi_h)\colon  \mathrm{Jac}(C) \longrightarrow \mathrm{Jac}(C_f)\times \mathrm{Jac}(C_g)\times \mathrm{Jac}(C_h)$ defined over $K$. We will show that it is surjective as well. Recall that the Prym variety $\mathrm{Prym}(C/C_f)$ is the kernel of $\psi_f$, because the function field $K(C)/K(C_f)$ is ramified at places dividing $h(x)$. and we have $\mathrm{Jac}(C) \sim \mathrm{Jac}(C_f) \times \mathrm{Prym}(C/C_f)$, which is defined over $K$. By \cite[page 346, statement (c)]{MUMFORD1974325}, we have an isomorphism $\mathrm{Prym}(C/C_f) \cong \mathrm{Jac}(C_g) \times \mathrm{Jac}(C_h)$. Hence, one obtains an isogeny between $\mathrm{Jac}(C)$ and $\mathrm{Jac}(C_f) \times \mathrm{Jac}(C_g) \times \mathrm{Jac}(C_h)$ over $\overline{K}$; in particular, this map is surjective over $\overline{K}$. This implies that $\dim(\mathrm{Jac}(C))=\dim(\mathrm{Jac}(C_f))+\dim(\mathrm{Jac}(C_g))+\dim(\mathrm{Jac}(C_h))$, so $g(C)=g(C_f)+g(C_g)+g(C_h)$, where $g(X)$ denotes the genus of the curve $X$. We use a slight abuse of notation with $g$ being a genus and a polynomial, but we believe that it is clear in the context which one is which. Given a smooth projective curve $X$, we denote by $X^k$ the product of $k$ copies of $X$, and by $X^{(k)}$ the $k$-th symmetric power of $X$. By using \cite[Theorem 5.1, Corollary 7.5]{Milne-JV}, we have the following commutative diagram.
\begin{equation*}
    \begin{tikzcd}[column sep = 10em]
        C^{g(C)} \arrow[d] \arrow[r, "\pi_f^{g(C_f)}\times\pi_g^{g(C_g)}\times\pi_h^{g(C_h)} "] &C_f^{g(C_f)} \times C_g^{g(C_g)} \times  C_h^{g(C_h)} \arrow[d] \\
        C^{(g(C))} \arrow[d] \arrow[r, "\pi_f^{g(C_f)}\times\pi_g^{g(C_g)}\times\pi_h^{g(C_h)} "] &C_f^{(g(C_f))} \times C_g^{(g(C_g))} \times  C_h^{(g(C_h))} \arrow[d] \\
        \Jac(C) \arrow[r, "\psi"] &\Jac(C_f) \times \Jac(C_g) \times \Jac(C_h)
    \end{tikzcd}
\end{equation*}
The upper right arrow map is clearly surjective as $g(C)=g(C_f)+g(C_g)+g(C_h)$, and all vertical maps are also surjective, hence, the bottom right map ($\psi$) is also surjective (over $\overline{K}$). Now,  \cite[Proposition 7.1]{Milne-AV} implies that $\psi$ is an isogeny. Since $\psi$ is defined over $K$, we have proven the following lemma.

\begin{lemma}\label{lem:isogeny-fibre-product}
The map $\psi\colon \Jac(C)\longrightarrow \Jac(C_f)\times \Jac(C_g)\times \Jac(C_h)$ is an isogeny defined over $K$.
\end{lemma}

\subsection{Higher genus - our strategy} \label{subsec:higher-genus-strategy}

Lemma~\ref{lem:isogeny-fibre-product} gives an idea on how to look for curves of any genus $g$ of some prescribed rank by descending to the same question but for smaller ranks of curves of smaller genus.

Given our strategy in Section~\ref{sec:genus2}, we can try to consider twists of hyperelliptic curves $C_f$, $C_g$, and $C_h$. However, their equations are related. If we choose to twist $C_f$ by $d$, and $C_g$ to twist by $m$ (we also allow $m=1$, i.e., preserve $C_g$), then in our construction, we need to consider the twist of $C_h$ by $dm$. Therefore, even though we can make quadratic twists of $C_f$ and $C_g$ to be independent, the quadratic twist on $C_h$ is not anymore independent. It will happen only in the case when $m=d$, but then the twists of $C_f$ and $C_g$ are not independent. Hence, to study the $\Rank(J(C)(K))$ as the sum of the ranks of Jacobians the twists of $C_f$, $C_g$, and $C_h$ is an extremely difficult problem in general because we need to control simultaneously dependent twists of different curves. 

Thus, our idea, which makes this strategy more efficient, is to make sure that one of the curves has genus 0, and then we can simply ignore its Jacobian. We achieve it by having a concrete relation between polynomials $f$ and $g$; namely, we require $f(x)=\gamma xg(x)$, where $\gamma\in K$ is a non-zero constant, and $g(0)\neq 0$ so that we consider smooth curves. Then, $C_h\colon y^2=\gamma x$ is isomorphic to $\mathbb{P}^1$ and will not contribute to the rank of the Jacobian. 

Furthermore, in this case, we obtain that the fibre product curve $C$ is a hyperelliptic curve and we get its explicit equation. Looking at the compositum of the function fields of $C_f$ and $C_g$, we have 
$K(C)=L(\sqrt{f(x)}, \sqrt{g(x)})=L(\sqrt{x}, \sqrt{g(x)})$. Now, we can identify the field $L(\sqrt{x})$ with $K(t)$ by putting $t=\sqrt{x}$, so it can also be seen as the function field of $\mathbb{P}^1$. Under this identification, we have that $L(\sqrt{x}, \sqrt{g(x)})\cong K(t,\sqrt{g(t^2)})$. The curve $C:y^2=g(x^2)$ has the function field $K(t,\sqrt{g(t^2)})$, hence $C$ is isomorphic to the fibre product of $C_f$ and $C_g$, and we may identify them. \\

We introduce some notation we will use in our construction.
Let $K$ be a number field, $f=x^{n}+a_{n-1}x^{n-1}+\cdots+a_1x+a_0\in \Oh_K[x]$ a monic polynomial of degree $n$, and $ d,m\in \Oh_K\backslash\{0\}$. Define
$$f_d(x)=x^n+da_{n-1}x^{n-1}+\cdots+d^{n-1}a_1x+d^na_0,$$
i.e., the polynomial such that $f_d(dx)=d^nf(x)$.

\begin{remark} \label{rem:isogeny-genus-g}
The fibre product construction in this case can be seen also very explicitly using the equations of the curves and the maps between them. 
Consider the following three hyperelliptic curves
\begin{align*}
C &\colon my^2=f_d(x^2); \\
C_1 &\colon my^2=xf_d(x); \\
C_2 &\colon my^2=f_d(x).
\end{align*}
The curve $C$ admits two degree 2 morphisms
\begin{align*}
\psi_1\colon C\longrightarrow C_1,&\;\; \psi_1(x,y)=(x^2,xy); \\
\psi_2\colon C\longrightarrow C_2,&\;\; \psi_2(x,y)=(x^2,y).
\end{align*}
We know (by what we proved in \S\ref{subsec:fibre-product}) that $C$ is the fibre product of $C_1$ and $C_2$ over $\mathbb{P}^1$ and  $\psi=({\psi_1}_*,{\psi_2}_*)\colon \Jac(C)\longrightarrow \Jac(C_1)\times \Jac(C_2)$ is an isogeny defined over $K$.

One can note that $C$ has a group of automorphisms that contains $\Z/2\Z\times \Z/2\Z$ as the maps $(x,y)\mapsto (\pm x,\pm y)$ are automorphisms of $C$. And the three quotients of $C$ by these nontrivial automorphisms are precisely $C_1$, $C_2$, and $\mathbb{P}^1$. Hence, we could use the work of Kani and Rosen \cite{KR89} to conclude that $\Jac(C)\sim \Jac(C_1)\times \Jac(C_2)$.

Additionally, there is another way to check that $\psi$ is an isogeny, as suggested by Maarten Derickx. We can note that the pull-backs $\psi_1^*(\Jac(C_1))\subseteq \Jac(C)$ and $\psi_2^*(\Jac(C_2))\subseteq \Jac(C)$  have $0$-dimensional intersection inside $\Jac(C)$. This can be checked by pulling-back the holomorphic differentials $\frac{dx}{y},\ldots\frac{x^{\lfloor\frac{n}{2}\rfloor-1}dx}{y}$ from $\Jac(C_1)$ and $\frac{dx}{y},\ldots\frac{x^{\lfloor\frac{n-1}{2}\rfloor-1}dx}{y}$ from $\Jac(C_2)$, which give the full basis of the holomorphic differentials $\frac{dx}{y},\ldots\frac{x^{n-2}dx}{y}$ on $\Jac(C)$.
\end{remark}

Finally, we can explain what is our strategy to construct hyperelliptic curves of any genus and fixed rank between 0 and 2 (or conditional ranks under the existence of a curve of a specific shape and rank). Note that, depending on the parity of $n$, after a linear change of variables ($x\mapsto dx$), we can simplify the equations for $C_1$ and $C_2$

\begin{align} 
    C_1\cong dm y^2=xf(x),\; C_2\cong m y^2=f(x)     & \text{ if } 2 \mid n, \label{eq:cases-even} \\
    C_1\cong m y^2=xf(x),\; C_2\cong dm y^2=f(x)     & \text{ if } 2\nmid n. \label{eq:cases-odd}
\end{align}

In general, we consider two twists of curves $C_1$ and $C_2$, but after specifying $m$, the other twist is independent of $m$ (by choosing an appropriate $d$). This is essential for the proof of Theorem~\ref{thm:genus-g-twists-rank+1}. 

However, for the theorems before, we can just consider the case when $m=d$, when, one of the curves is isomorphic to a fixed one, given by an odd degree model, whereas an even degree model becomes a twist by $d$. In this case, this gives a clear strategy to construct infinitely many curves of genus $g$ of a fixed rank $r$; we just need to find a specific one satisfying some mild conditions. 

\subsection{Our constructions}\label{subsec:constructions}

Before we proceed, we recall a few definitions and results from previous literature.
\begin{definition}[Definition 1.2 and 2.2 of \cite{KM25}]\label{def:B-generic}
    Let $C: y^2 = (x-a_1) \cdots (x-a_{2g+1})$ be an odd degree hyperelliptic curve over a number field $K$. Given a place $\omega$ of $K$ and a pair of distinct indices $1 \leq i, j \leq 2g+1$, we say that $\omega$ is multiplicative of type $\{i,j\}$ if for every pair of distinct indices $1 \leq k, \ell \leq 2g+1$, we have
    \begin{equation}
        \mathrm{val}_\omega(a_k - a_\ell) = \begin{cases}
            1 &\text{ if } \{k,\ell\} = \{i,j\}, \\
            0 &\text{ otherwise}.
        \end{cases}
    \end{equation}
    Given a positive number $B > 0$, we say that $C$ is $B$-generic if there exist $4g^2 + 4g - 1$ distinct finite places 
    \begin{equation*}
        \begin{cases}
            \omega_1, \omega_2, &\cdots, \omega_{4g^2 + 2g - 1}, \\
            \omega'_1, \omega'_2, &\cdots, \omega'_{2g}.
        \end{cases}
    \end{equation*}
    which satisfy the following four conditions.
    \begin{itemize}
        \item For every $1 \leq i, j \leq 2g$, the place $\omega_{(i-1)2g + j}$ is multiplicative of type $\{j, 2g+1\}$.
        \item For every $1 \leq i \leq 2g - 1$, the place $\omega_{4g^2 + i}$ is multiplicative of type $\{i, i+1\}$.
        \item For every $1 \leq i \leq 4g^2 + 2g - 1$, we have $
        \#\Oh_K/(\omega_i) \geq B$.
        \item For every $1 \leq i \leq 2g$, the place $\omega'_i$ is multiplicative of type $\{i, 2g + 1\}$.
    \end{itemize}
\end{definition}

\begin{definition}\label{def:B-uber-generic} 
Using the notation of Definition~\ref{def:B-generic}, we say that $C$ is $B$-\emph{\"{u}ber-generic} if the following conditions are satisfied: There are $g^2(2g+1)$ distinct places of $K$ such that for each pair $1\leq k<\ell\leq 2g+1$ we have for exactly $2g+1$ places $\omega$ that 
\begin{equation}
        \mathrm{val}_\omega(a_k - a_\ell) = \begin{cases}
            1 &\text{ if } \{k,\ell\} = \{i,j\}, \\
            0 &\text{ otherwise},
        \end{cases}
    \end{equation}
and for all places $\omega$, and all $a_i$, we have $\mathrm{val}_w(a_i)=0$ and $
        \#\Oh_K/(\omega) \geq B$.
\end{definition}

We focus on understanding the Mordell-Weil rank of the family of quadratic twists of hyperelliptic curves that are $B$-\"uber-generic. Before doing so, we recall the following condition on $C$ from \cite[Condition 5.1]{BM25} and \cite[Proposition 4.3]{MS24}. Once again, we would like to thank Adam Morgan for suggesting the importance of the third condition provided below.
\begin{condition}[Condition 5.1 from \cite{BM25} and Proposition 4.3 from \cite{MS24}] \label{condition:Morgan}
    Let $C$ be a hyperelliptic curve over $K$ with $K(\mathrm{Jac}(C)[2]) = K$. Then $C$ satisfies the following three conditions.
    \begin{itemize}
        \item The invariant subspaces of $\mathrm{Jac}(C)[2]$ with respect to $\mathrm{Gal}(K(\mathrm{Jac}(C)[4])/K)$ are $0$ or $\mathrm{Jac}(C)[2]$.
        \item The set of elements of $\mathrm{End}$ which commute with every element of $\mathrm{Gal}(K(\Jac(C)[4])/K)$ via multiplication is isomorphic to $\mathbb{F}_2$.
        \item There exists a place $\omega$ of $K$ such that $C$ has a single nodal singularity at $\omega$. (In particular, the latter condition is satisfied if there exists a place $\omega$ and a tuple $(i,j)$ such that $\mathrm{val}_\omega(a_i - a_j) = 1$ and $\mathrm{val}_\omega(a_k - a_l) = 0$ for any other $k,l$ such that $\{k,l\} \neq \{i,j\}$.)
    \end{itemize}
\end{condition}

\begin{remark} \label{remark:generic->Condition5.1}
    We note that if $C$ is $B$-generic, then there exists a square-free element $d \in \Oh_K^\times$ such that $C_d$ satisfies Condition \ref{condition:Morgan}. To see this, we note that \cite[Proposition 8.10]{BM25} shows that any hyperelliptic curve $C: y^2 = (x-a_1)\cdots(x-a_{2g+1})$ satisfying the following two conditions 
    \begin{itemize}
        \item There exists a set of places $\{\omega_1, \omega_2, \cdots, \omega_{2g}\}$ such that for each $1 \leq i \leq 2g$, $\mathrm{val}_{\omega_i}(a_i - a_{2g+1}) = 1$ whereas $\mathrm{val}_{\omega_i}(a_k - a_l) = 0$ for any other $k,l$ such that $\{k,l\} \neq \{i,2g+1\}$.
        \item There exists an element $\zeta \in \mathrm{Sel}_2(\Jac(C)) \setminus \Jac(C)[2]$ such that the restriction of $\zeta$ over $H^1(K_{\omega_i}, \Jac(C)[2])$ is unramified for all $1 \leq i \leq 2g$.
    \end{itemize}
    satisfies the first two requirements in Condition \ref{condition:Morgan}. By \cite[Theorem 3.8]{KM25}, there exists a square-free element $d \in \Oh_K^\times$ such that the quadratic twist $C_d$ of a $B$-generic hyperelliptic curve $C$ satisfies the two bullet points mentioned above. By the condition on $a_i$'s, one can easily check that $C$ has a single nodal singularity at the place $\omega_1$, which satisfies the last condition from Condition \ref{condition:Morgan}.
\end{remark}

\begin{remark} \label{remark:weil->Condition5.1}
    Let $C$ be a hyperelliptic curve over $\mathbb{Q}$. One can check the first two conditions from Condition \ref{condition:Morgan} by computing the Weil pairing $e_2: \Jac(C)[2] \times \Jac(C)[2] \to \mu_2$. Suppose $\langle P_1, P_2, \cdots, P_{2g}\rangle$ is a basis for the set $\Jac(C)[2]$. If for every pair $1 \leq i < j \leq 2g$ we have $e_2(P_i,P_j) = -1$, then the first two conditions hold by \cite[Proposition 2.15]{BM25}. 

    For the special case when $C$ is given by the Weierstrass model $C: y^2 = (x-b_1)\cdots(x-b_{4k+3})$ for some positive integer $k$, we can combine \cite[Corollary 7.6]{Yu16} and \cite[Theorem 2.14]{Smith2} to ensure that at least $50\%$ of Jacobians of quadratic twists of $C$ have rank $0$. In the case when $C$ is not of the form above, we can replace the third condition in Condition \ref{condition:Morgan} with the requirement that there exists a single square-free element $d \in \mathbb{Q}^\times$ such that $\dim_{\mathbb{F}_2} \mathrm{Sel}_2(\mathrm{Jac}(C_d)/\mathbb{Q}) \equiv 0 \mod 2$. This condition ensures that there is a positive density of Jacobians of rank $0$, as shown in the discussion following after \cite[Case 2.13]{Smith2} and the main theorem of \cite{BM25}.
\end{remark}

\begin{theorem} \label{thm:Condition5.1}
    Let $C: y^2 = (x-a_1)\cdots(x-a_{2g+1})$ be a hyperelliptic curve over a number field $K$ that satisfies Condition \ref{condition:Morgan}. Then there are infinitely many square-free elements $d \in \Oh_K^\times$ such that $\mathrm{Rank}(\Jac(C_d)(K)) = 0$.
\end{theorem}
\begin{proof}
    The proof of \cite[Corollary 6.15]{KM25} and the first paragraph of the proof of \cite[Proposition 4.3]{MS24} implies that all the conditions in \cite[Case 2.13]{Smith2} are satisfied, and that a positive density of square-free elements $d \in \Oh_K$ satisfies $\dim_{\mathbb{F}_2} \mathrm{Sel}_{2}(\mathrm{Jac}(C_d)/K) \equiv 0 \mod 2$. Hence, we can apply \cite[Theorem 2.14]{Smith2} to show that there are infinitely many quadratic twists of $C$ such that $\mathrm{Rank}(\Jac(C_d)/K) = 0$.
\end{proof}

\begin{corollary}
    For any $g\geq 2$, there are infinitely many hyperelliptic curves $C$ of genus $g$ over a number field $K$ such that $\mathrm{Rank}(\mathrm{Jac}(C)(K)) = 0$.
\end{corollary}
\begin{proof}
    We apply Theorem \ref{thm:Condition5.1} to any $B$-generic hyperelliptic curve $C$ of genus $g$, which can easily be constructed using the Chinese Remainder Theorem.
\end{proof}

\begin{theorem}\label{thm:genus-g-twists-rank+0}
Let $g\geq 3$ be a positive integer and denote $k=\lfloor\frac{g+1}{2}\rfloor$. Assume there are $b_1,\ldots,b_{2k+1}\in \Oh_K$ so that the hyperelliptic curve 
$$X\colon y^2=(x-b_1)\cdots(x-b_{2k+1})$$
has $\Rank(\Jac(X)(K))=r$. Suppose for each specified value of $g$ as mentioned below, the curve $Y_1$ satisfies Condition \ref{condition:Morgan}:
\begin{itemize}
    \item $g = 2k$, and $Y_1: y^2 = x(x-b_1)\cdots(x-b_{2k+1})$.
    \item $g = 2k-1$, and $Y_1: y^2 = (x-(b_1-b_{2k+1})) \cdots (x-(b_{2k}-b_{2k+1}))$.
\end{itemize}
Then, there are infinitely many hyperelliptic curves of genus $g$ whose Jacobian over $K$ has rank $r$.
\end{theorem}

\begin{proof}
Suppose $g=2k$, Let $Y_d \colon dy^2=x(x-b_1)\cdots(x-b_{2k+1})$ be a quadratic twist of $Y_1$ by $d$. By Theorem \ref{thm:Condition5.1}, there are infinitely many such quadratic twists $Y_d$ such that $\mathrm{Rank}(\Jac(Y_d)(K)) = 0$.
Hence, there are infinitely many genus $g$ curves 
$$C_d\colon dy^2=(x^2-db_1)\cdots (x^2-db_{2k+1})$$
of rank $r$ as $\Jac(C_d)$ is isogenous to $\Jac(X)\times \Jac(Y_d) $ by Remark~\ref{rem:isogeny-genus-g}.

Similarly, if $g=2k-1$, then, 
$$X\cong X'\colon y^2=x(x-(b_1-b_{2k+1}))\cdots(x-(b_{2k}-b_{2k+1})).$$ By Theorem \ref{thm:Condition5.1}, there are infinitely many twists 
$$Y_d\colon dy^2=(x-(b_1-b_{2k+1}))\cdots(x-(b_{2k}-b_{2k+1}))$$ of rank 0. By Remark~\ref{rem:isogeny-genus-g}, there are infinitely many curves 
$$C_d\colon dy^2=(x^2-d(b_1-b_{2k+1}))\cdots(x^2-d(b_{2k}-b_{2k+1}))$$
of genus $g$ such that $\Rank(\Jac(C_d)(K))=r$.
\end{proof}

\begin{remark}
    We note that the affine models for curves $Y_1$ in both cases $g = 2k$ and $g = 2k-1$ are defined in terms of even degree polynomials in $x$. For computational purposes, we rewrite the models for $Y_1$ in terms of odd degree polynomials in $x$. For the case when $g = 2k$, let $\mathcal{B} := \prod_{i=1}^{2k+1} b_i$. Then the equation defining $Y_1$ can be rewritten as
    \begin{equation*}
        Y_1 \cong Y_1' : y^2 = \left(x + \frac{\mathcal{B}}{b_1} \right) \cdot \left(x + \frac{\mathcal{B}}{b_2} \right) \cdots \left(x + \frac{\mathcal{B}}{b_{2k+1}} \right).
    \end{equation*}
    If $g = 2k-1$, let $\mathcal{B} := \prod_{i=2}^{2k}(b_i - b_1)$. Then we can rewrite the equation defining $Y_1$ as
    \begin{equation*}
        Y_1 \cong Y_1' : y^2 = \left(x + \frac{\mathcal{B}'}{b_2 - b_1} \right) \cdot \left(x + \frac{\mathcal{B}'}{b_3 - b_1} \right) \cdots \left(x + \frac{\mathcal{B}'}{b_{2k} - b_1} \right).
    \end{equation*}
Then, in our applications, see \S\ref{subsec:examples-3456-Q}, we check Condition~\ref{condition:Morgan} directly for these odd degree models using Magma.   
\end{remark}

We can construct such examples in a different way, i.e., coming of the hyperelliptic curves, whose Jacobians have very few $K$-rational 2-torsion points, using the result of Yu, \cite[Theorem 1.2]{Yu19}.

\begin{theorem}\label{thm:any-g-rank-r-Yu-twists}
Let $g\geq 3,r\geq 0$ be integers and $K$ a number field.
\begin{enumerate}
\item Assume that $g=2k$ for some $k\in\N$ and that there exists a hyperelliptic curve $X:y^2=f(x)$ such that $\Rank(\Jac(X)(K))=r$, where $f\in \Oh_K[x]$, $\deg(f)=2k+1$, is a monic polynomial whose splitting field has Galois group isomorphic to $S_{2k+1}$. Suppose in addition that there exists a square-free element $m \in \mathcal{O}_K$ such that the curve $Y_m: my^2 = xf(x)$ satisfies $\dim_{\mathbb{F}_2}\mathrm{Sel}_2(\mathrm{Jac}(Y_m)/K) \equiv 0 \mod 2$. Then, there are infinitely many genus $g$ curves $C_d: dy^2 = f_d(x^2)$ of rank $r$, where $d\in \Oh_K$.
\item Assume that $g=2k-1$ for some $k\in\N$ and that there exists a hyperelliptic curve $X:y^2=x(x-1)f(x)$ such that $\Rank(\Jac(X)(K))=r$, where $f\in \Oh_K[x]$, $\deg(f)=2k-1$, is a monic polynomial whose splitting field has Galois group isomorphic to $S_{2k-1}$. Suppose in addition that there exists a square-free element $m \in \mathcal{O}_K$ such that the curve $Y_m: my^2 = (x-1)f(x)$ satisfies $\dim_{\mathbb{F}_2}\mathrm{Sel}_2(\mathrm{Jac}(Y_m)/K) \equiv 0 \mod 2$. Then, there are infinitely many genus $g$ curves $C_d: dy^2 = (x^2-d^2)f_d(x^2)$ of rank $r$, where $d\in \Oh_K$.
\end{enumerate}  
\end{theorem}

\begin{proof}
\begin{enumerate}
\item\label{proof:Yu-first-case} We note that the existence of a square-free element $m \in \mathcal{O}_K$ such that $Y_m: my^2 = xf(x)$ satisfies $\dim_{\mathbb{F}_2}\mathrm{Sel}_2(\mathrm{Jac}(Y_m)/K) \equiv 0 \mod 2$ implies that the disparity $\delta(\mathrm{Jac}(Y_1)/K) \neq -\frac{1}{2}$ appearing in the statement of \cite[Theorem 1.2]{Yu19}. The assumptions assure that we can apply \cite[Theorem 1.2]{Yu19} to ensure that there are infinitely many quadratic twists $Y_d\colon dy^2=xf(x)$ of $Y\colon y^2 = xf(x)$ such that $\mathrm{Rank}(\Jac(Y_d)(K)) = 0$. Then, the statement follows from Remark~\ref{rem:isogeny-genus-g} which gives that $\Jac(C_d)$ is isogenous to $\Jac(X)\times \Jac(Y_d)$ over $K$, and the argument that there are infinitely many curves that are non-isomorphic among them is as before. 
\item The proof is the same as in \eqref{proof:Yu-first-case}, and now we apply \cite[Theorem 1.2]{Yu19} to conclude that there are infinitely many quadratic twists $Y_d\colon dy^2=(x-1)f(x)$ of $Y\colon y^2 = (x-1)f(x)$ such that $\mathrm{Rank}(\Jac(Y_d)(K)) = 0$.
\end{enumerate} 
\end{proof}

\begin{theorem}\label{thm:genus-g-twists-rank+1}
    Let $k>1$ be a positive integer. There is $B>0$ such that the following holds. Let $X: y^2 = (x-a_1)\cdots(x-a_{2k+1})$ be a hyperelliptic curve over a number field $K$ that $X$ is $B$-\"{u}ber-generic and it has a twist $X_m: my^2 = (x-a_1)\cdots(x-a_{2k+1})$ such that $\Rank(\Jac(X_m)(K))=r$. Then there are infinitely many hyperelliptic curves $C/K$ of genus $g$, where $g\in\{2k-1,2k\}$ satisfying $\Rank(\Jac(C)(K))=r+1$. 
\end{theorem}
\begin{proof}
The easier case is when $g=2k-1$. We may assume that $a_{2k+1}=0$ (otherwise just replace $x-a_{2k+1}$ by $x$). We aim to create the curve $C$ as a fibre product of $X_m$ with $Y_{dm}: {dm}y^2 = (x-a_1)\cdots(x-a_{2k})$ (we are in the situation of Equation~\eqref{eq:cases-even}, but with changed twisting numbers, here $m$ corresponds to $dm$ from Equation~\eqref{eq:cases-even}, and then $dm$ here corresponds to $d^2m$ equal to $m$ modulo squares from Equation~\eqref{eq:cases-even}). We choose $B>0$ from \cite[Theorem 1.3]{KM25} for genus $k-1$. Our assumption of being $B$-\"{u}ber-generic clearly implies that $Y: y^2 = (x-a_1)\cdots(x-a_{2k})$ is $B$-generic. Hence, by \cite[Theorem 1.3]{KM25}, there are infinitely many twists $Y_{dm}: {dm}y^2 = (x-a_1)\cdots(x-a_{2k})$ of $Y$ such that $\mathrm{Rank}(\Jac(Y_{dm})(K)) = 1$. Then, for $C:my^2=(x^2-da_1)\cdots(x^2-da_{2k})$, by Remark~\ref{rem:isogeny-genus-g}, we have that $\Jac(C)$ is isogenous to $\Jac(X_m)\times \Jac(Y_{dm})$ over $K$. Hence, $\Rank(\Jac(C)(K))=r+1$, and the argument that there are infinitely many non-isomorphic curves among them is as before.

Now, we consider the case $g=2k$.  We may assume now $a_1\cdots a_{2k+1}\neq 0$. We aim to construct the curve $C$ as a fibre product of $X_m$ with $Y_{dm}: dmy^2 = x(x-a_1)\cdots(x-a_{2k+1})$ (we are in the situation of Equation~\eqref{eq:cases-odd}, but, again, with changed twisting numbers, here $m$ corresponds to $dm$ from Equation~\eqref{eq:cases-odd}, and then $dm$ here corresponds to $d^2m$ equal to $m$ modulo squares from Equation~\eqref{eq:cases-odd}). We choose $B>0$ from \cite[Theorem 1.3]{KM25} for genus $k$. Now, we prove that $Y: y^2 = x(x-a_1)\cdots(x-a_{2k+1})$ is $B$-\"{u}ber-generic, so we need to use an odd-degree model for $Y$. The change of variables $x\mapsto\frac{1}{x}$, induces that $Y$ is isomorphic over $K$ to 
$$Z:y^2=\left(x+\frac{A}{a_1}\right)\cdots\left(x+\frac{A}{a_{2k+1}}\right),$$
where $A=a_1\cdots a_{2k+1}$. Now, the assumption that $X$ is $B$-\"{u}ber-generic gives that for any $i,j$ and valuation $w$
$$\mathrm{val}_w\left(\frac{A}{a_i}-\frac{A}{a_j}\right)=\mathrm{val}_w\left(\frac{A}{a_ia_j}(a_j-a_i)\right)=\mathrm{val}_w(a_j-a_i),$$
so $Z$ is $B$-generic. Therefore, we can apply \cite[Theorem 1.3]{KM25} to conclude that there are infinitely many twists $Z_{dm}$ of $Z$ such that $\Rank(\Jac(Z_{dm})(K))=1=\Rank(\Jac(Y_{dm})(K))$ for the corresponding twist $Y_{dm}$ of $Y$. Now, we conclude as in the previous case.
\end{proof}

\begin{remark}
    Using the same idea of the proof that $X$ being $B$-\"uber-generic implies $Y_d$ being $B$-generic, we can use Remark \ref{remark:generic->Condition5.1} and \cite[Theorem 2.14]{Smith2} to prove that if $X$ is $B$-\"uber-generic and $\mathrm{Rank}(\Jac(X)(K)) = r$, then there are infinitely many hyperelliptic curves of genus $g$ whose Jacobian over $K$ has rank $r$.
\end{remark}

\begin{theorem}\label{thm:any-g-rank-2}
Let $g\geq 3$ be an integer and $K$ a number field. There are infinitely many hyperelliptic curves $C$ of genus $g$ such that $\Rank(\Jac(C)(K))=r$ for $r = 1$ and $2$.    
\end{theorem}
\begin{proof}
    By \cite{KM25}, we can always find a quadratic twist of a $B$-\"uber-generic curve $X$ of genus $\lfloor\frac{g+1}{2}\rfloor$ whose Jacobian over $K$ has rank $1$. We then apply Theorem \ref{thm:genus-g-twists-rank+0} ($+$ Remark \ref{remark:generic->Condition5.1}) and Theorem \ref{thm:genus-g-twists-rank+1} to conclude.
\end{proof}

\subsection{Examples over $\Q$}\label{subsec:examples-3456-Q}
Now we prove the existence of infinitely many curves over $\Q$ of genus $g$ whose Jacobian over $\Q$ has rank $r$ for $g\in\{3,4\}$ and $1\leq r\leq 4$ and $g\in\{5,6\}$ and $1\leq r\leq 3$. Even though the existence of the rank 1 and 2 follows from Theorem~\ref{thm:any-g-rank-2}, we still decided to include these examples.

We will always use Theorem~\ref{thm:genus-g-twists-rank+0}. We will find a hyperelliptic curve of the shape 
$$X\colon y^2=(x-b_1)\cdots(x-b_{2k+1})$$
such that $\Rank(\Jac(X)(\Q))=r$ and check that the conditions for the curve $Y_1$ from Theorem~\ref{thm:genus-g-twists-rank+0} using Remark~\ref{remark:weil->Condition5.1}. We note that only in the case for genus three curves, we do not need to do any checks because in this case $Y_1$ is an elliptic curve and then the result about twisting comes automatically from \cite{Smith3}.

\subsubsection{Genus three and four}\label{subsubsec:genus34} We present a table of genus two curves $X\colon y^2=(x-b_1)\cdots(x-b_{5})$, which, by Remark \ref{remark:weil->Condition5.1} and Theorem~\ref{thm:genus-g-twists-rank+0}, prove the existence of infinitely many hyperelliptic curves $C/\Q$ of genus three and four such that $\Rank(\Jac(C)(\Q))=r$, for the following $r$ from the table.

\begin{center}
    \begin{tabular}{|c|c|c|c|c|c|}
    \hline
      $r$  & $b_1$ & $b_2$ & $b_3$ & $b_4$ & $b_5$ \\
      \hline
      1  & 1 & 2 & 3 & 4 & 10 \\
      \hline
      2  & 1 & 2 & 3 & 4 & 37 \\
      \hline
      3  & -49 & -4 & 1 & 43 & 51 \\
      \hline
      4  & -49 & 1 & 2 & 28 & 51 \\
      \hline
    \end{tabular}
    \label{tab:g=2:ranks-0-to-4}
\end{center}

\subsubsection{Genus five and six}\label{subsubsec:genus56} We present a table of genus three curves $X\colon y^2=(x-b_1)\cdots(x-b_{7})$, which, by Remark \ref{remark:weil->Condition5.1} and Theorem~\ref{thm:genus-g-twists-rank+0}, prove the existence of infinitely many hyperelliptic curves $C/\Q$ of genus five and six such that $\Rank(\Jac(C)(\Q))=r$, for the following $r$ from the table.

\begin{center}
    \begin{tabular}{|c|c|c|c|c|c|c|c|}
    \hline
      $r$  & $b_1$ & $b_2$ & $b_3$ & $b_4$ & $b_5$ & $b_6$ & $b_7$ \\
      \hline
      1  & 1 & 2 & 3 & 4 & 5 & 6 & 7 \\
      \hline
      2  & 1 & 2 & 3 & 4 & 6 & 7 & 11 \\
      \hline
      3  & 1 & 2 & 3 & 4 & 5 & 14 & 21 \\
      \hline
    \end{tabular}
    \label{tab:g=3:ranks-0-to-3}
\end{center}

\begin{remark}
Even though it sounds appealing to use Theorem~\ref{thm:genus-g-twists-rank+1} to construct examples, we note that the $B$-generic and $B$-\"{u}ber-generic conditions are less practical to use. Namely, a careful inspection on the condition on $B$ from \cite{KM25} shows that $B>9600$ for genus two curves, so we will need to use many large primes to construct such $B$-generic or $B$-\"{u}ber-generic curves. Moreover, solving the Chinese Remainder Theorem system will result in large coefficients of the curve $C$. Hence, we have not tried to use $B$-generic or $B$-\"{u}ber-generic curves to get concrete examples. 

Also, Theorem~\ref{thm:any-g-rank-r-Yu-twists} is more tricky for applications because it is much easier to do explicit descent for elliptic curves defined by a polynomial which splits completely, rather than of those whose Galois group is as maximal as possible. Hence, we have not tried to use this theorem to construct examples either.
\end{remark}

\section{An answer to our question}\label{sec:challenges}

If the Jacobian of a curve is decomposable as a product of Jacobians of curves, then we reduce the problem of computing the rank to smaller genus curves, which is, in principle, easier, even though, as we have seen, it also contains its own subtleties. Since we gave a construction of a decomposable Jacobian of a fixed positive rank, we posed a challenge in the previous version of this article as stated in June 2025 to find such examples of curves $C$ such that $\Jac(C)$ is absolutely simple. The breakthrough work of  Koymans and Morgan \cite{KM25} resolves the following question (and in fact for any genus $g$). 

\begin{question}
Is there an infinite family of curves $C/\Q$ of genus two or three with absolutely simple Jacobian such that the rank of $\Jac(C)(\Q)$ is equal to $1$ (or some other fixed positive integer)?   
\end{question}

We keep the question as stated in its original form, as \cite[Theorem 1.1]{KM25} specifically cites the problem mentioned beforehand. We sincerely congratulate and thank them for solving the question stated above.

\subsection*{Acknowledgements} We would like to thank the Max Planck Institute for Mathematics for its generous support. We would like to also thank Lodha Mathematical Sciences Institute for its generous support and hospitality during the Thematic Programme on Arithmetic Statistics where substantial results of this paper were achieved. We would especially like to thank Maarten Derickx and Carlo Pagano for their advice and inspiring discussions about this project, and in particular, about motivating us to consider the general situation from \S\ref{subsec:fibre-product}, Jennifer Balakrishnan, Du\v{s}an Dragutinovi\'{c}, Peter Koymans, Adam Morgan, and Lazar Radi\v{c}evi\'{c} for numerous conversations and their enlightening comments about this work. We would also like to thank Jordan Ellenberg, Matija Kazalicki, Jef Laga, Filip Najman, Oana Padurariu, Alex Smith, and Wadim Zudilin for helpful discussions.

\bibliographystyle{alpha}
\bibliography{draft}

\end{document}